\documentclass[12pt, a4paper]{amsart}
\usepackage{amssymb,fullpage}
\usepackage{amsmath,esint}
\usepackage[colorlinks, allcolors=RedViolet]{hyperref}
\usepackage{colortbl}
\usepackage[dvipsnames]{xcolor}

\newtheorem{theorem}[equation]{Theorem}
\newtheorem{lemma}[equation]{Lemma}

\newtheorem{corollary}[equation]{Corollary}
\theoremstyle{definition}
\newtheorem{definition}[equation]{Definition}

\theoremstyle{remark}
\newtheorem{remark}[equation]{Remark}
\numberwithin{equation}{section}

\newcommand{ \mr }{ \mathbb{R} }
\newcommand{ \R }{ \mathbb{R} }

\newcommand{\tphi}{{\tilde\phi}}

\newcommand{\loc}{{\operatorname{loc}}}
\renewcommand{\epsilon}{\varepsilon}
\renewcommand{\phi}{\varphi}
\renewcommand{\le}{\leqslant}
\renewcommand{\ge}{\geqslant}
\renewcommand{\leq}{\leqslant}
\renewcommand{\geq}{\geqslant}

\newcommand{\ainc}[1]{\hyperref[ainc]{{\normalfont(aInc){\ensuremath{_{#1}}}}}}
\newcommand{\adec}[1]{\hyperref[adec]{{\normalfont(aDec){\ensuremath{_{#1}}}}}}
\newcommand{\inc}[1]{\hyperref[inc]{{\normalfont(Inc){\ensuremath{_{#1}}}}}}
\newcommand{\dec}[1]{\hyperref[dec]{{\normalfont(Dec){\ensuremath{_{#1}}}}}}

\newcommand{\wMAe}[1]{\hyperref[wMAe]{{\normalfont(wMA){\ensuremath{_{#1}}}}}}

\begin{document}

\title{Higher integrability for parabolic systems with Orlicz  growth}

\author{Peter H\"ast\"o}

\address{Department of Mathematics and Statistics, FI-20014 University of Turku, Finland 
and
Department of Mathematics, FI-90014 University of Oulu, Finland}
\email{peter.hasto@oulu.fi}

\author{Jihoon Ok}
\address{Department of Applied Mathematics and the Institute of Natural Sciences, Kyung Hee University, Yongin 17104, Republic of Korea}
\email{jihoonok@khu.ac.kr}

\thanks{}





\begin{abstract}
We prove higher integrability of the spatial gradient of weak solutions to parabolic systems 
with $\phi$-growth, where $\phi=\phi(t)$ is a general Orlicz function. The parabolic 
systems need be neither degenerate nor singular. Our result is a generalized version of 
the one of J. Kinnunen and J. Lewis [Duke Math. J. 102 (2000), no. 2, 253--271] 
for the parabolic $p$-Laplace systems.
\end{abstract}

\maketitle


\section{Introduction}\label{sec1}

Higher integrability results for elliptic problems with Orlicz growth can be easily obtained from the ones for $p$-growth problems, hence they are well-known. On the other hand, higher integrability for parabolic problems with Orlicz growth is not simple and, as far as we know, no related result has been reported. The main difficulty is that  the lower- and the upper-bound of exponent of the 
Orlicz function $\phi$, which are denoted by $p$ and $q$ in this paper, may be too far away from each other to apply techniques used in the standard $p$-growth case. Specifically, in $p$-growth problems, known proofs use different techniques in the degenerate case ($p> 2$) and the singular case ($p<2$). However, in the general Orlicz setting, neither of these cases apply when $p<2<q$, since the problem 
has characteristics of both the singular and degenerate cases. 

We study regularity theory for second-order parabolic systems satisfying  a general growth condition of Orlicz type. Precisely,  we consider the following parabolic system:
\begin{equation}\label{maineq}
\partial_t u_i-\mathrm{div}(A_i(z,u,Du))=0\quad \text{in }\ \Omega_I=\Omega\times I\subset \R^n\times \R,\qquad i=1,\dots ,N,
\end{equation}
where $\Omega\subset\R^n$ ($n\geq 2$) is an open set, $I\subset\R$ is an interval, $z=(x,t)\in \Omega\times I$, $u=(u_1,\dots,u_N)\in \R^N$ and $Du$ is the spatial gradient of $u$, i.e., $Du=D_xu$. Here $A_i:\Omega_T\times \R^N\times \R^{Nn} \to \R^{n}$, $i=1,\dots,N$, satisfies
\begin{equation}\label{A-condition}
|A_i(z,u,\xi)| \leq \Lambda \frac{\phi(|\xi|)}{|\xi|}
\quad \text{and}\quad
\sum_{i=1}^N A_i(z,u,\xi)\cdot\xi_i  \geq \nu \frac{\phi(|\xi|)}{|\xi|} 
\end{equation}
for all $z\in\Omega_I$, $u\in \R^N$ and $\xi=(\xi_1,\dots,\xi_N)\in \R^{n}\times \dots\times \R^n$ and for some $0<\nu\le \Lambda$, and $\phi:[0,\infty)\to[0,\infty)$ is a weak $\Phi$-function which satisfies \ainc{p} and \adec{q} for some
\begin{equation}\label{pq-range}
\frac{2n}{n+2} < p  \leq q.
\end{equation}
We will introduce the definitions of weak $\Phi$-function, \ainc{} and \adec{} in the next section. We note that \ainc{p} and \adec{q} for some $1<p\leq q$ are equivalent to the $\nabla_2$  and $\Delta_2$ conditions, respectively, see \cite[Proposition~2.2.6]{HarH_book}. The lower bound $\frac{2n}{n+2}$ in \eqref{pq-range} is generally assumed in parabolic regularity theory, see \cite{DiBenedetto_book} and also \cite{KiLewis00}.

The prototype of \eqref{maineq} is the so-called parabolic $g$-Laplace system
\[
\partial_t u_i-\mathrm{div}\left(\frac{g(|Du|)}{|Du|}Du_i\right) =0,\qquad i=1,\dots,N.
\]
More generally, we may also consider coefficients:
\[
\partial_t u_i-\mathrm{div}\left(a(z,u)\frac{g(|Du|)}{|Du|}Du_i\right) =0,\quad i=1,\dots,N, 
\quad \text{where }\ 0<\nu\leq a(\cdot,\cdot)\leq \Lambda.
\]
Here, we may take $\phi(t)=\int_0^tg(s)ds$. In particular, when $g(t)=t^{p-1}$, 
this system  becomes the parabolic $p$-Laplace system.

The main result of this paper is to prove higher integrability of the gradient of a weak 
solution to the system \eqref{maineq} together with a reverse H\"older type estimate. 
The \textit{weak solution to \eqref{maineq}} 
with structure conditions \eqref{A-condition}--\eqref{pq-range}
is defined as a function $u\in L^\infty(I,L^2(\Omega,\R^N))\cap L^1(I,W^{1,1}(\Omega))$ 
with $
\phi(|Du|)\in L^1([0,T],L^1(\Omega))$ satisfying
\begin{equation}\label{weakform}
-\int_{\Omega_I} u_i \,\partial_t \zeta_i \, dz +\int_{\Omega_I} A_i(z,u,Du)\cdot D\zeta_i \,dz =0,\qquad i=1,\dots, N, 
\end{equation}
for all $\zeta=(\zeta_1,\dots,\zeta_N)\in C^\infty_0(\Omega_I,\R^N)$. 
We show that there exists a universal constant $\epsilon>0$ such that  
\[
\phi(|Du|)\in L^{1+\epsilon}_{\loc}(\Omega_I).
\]

Regularity theory for the parabolic $p$-Laplace systems, $p\neq 2$, was first systematically studied by 
DiBenedetto and Friedman, see \cite{DiBeFrie84,DiBeFrie85} and also the monographs \cite{DiBenedetto_book,DiBeGiaVe_book}. Later, $L^q$-regularity theory was established in \cite{AcerMin07, KiLewis00}. 
In particular, in \cite{KiLewis00}, Kinnunen and Lewis first proved higher integrability for 
parabolic $p$-Laplace systems. We further refer to \cite{Baroni13,Bogelein14,BoDuMin13,KuuMin12,KuuMin14,Misawa02} and related references for regularity results for parabolic $p$-Laplace systems.  

In the calculus of variations, partial differential equations with $p$-growth can be obtained as Euler-Lagrange equations of functionals with a $p$-growth condition that is related to the power function $t^p$. Hence we can naturally generalize $t^p$ to an Orlicz function, and a growth condition related to an Orlicz function is called the Orlicz-growth condition. Regularity results for elliptic equations with Orlicz growth, specifically $C^{\alpha}$- and $C^{1,\alpha}$-regularity, 
were first obtained by Lieberman \cite{Lieberman91}. 
Later, he generalized these results to parabolic systems with Orlicz growth \cite{Lieberman06}. 
We also refer to regularity results \cite{Baroni15,Cho18-1,CiaMa11,CiaMa14,DieEtt08,DieStroVer09,Verde11} and  \cite{BaLin17,Cho18-2,DieScharSchwa19,HwangLie15-1,HwangLie15-2,Yao19} for the 
elliptic and parabolic case with Orlicz growth, respectively. 

As mentioned above, we shall prove a higher integrability result for parabolic systems with Orlicz growth. The higher integrability is the most basic regularity property of weak solutions for elliptic/parabolic problems in divergence form, and is a crucial ingredient in studying regularity theory, see for instance \cite{Giusti_book}. It has been obtained first by Elcrat and Meyers  \cite{MeElcrat75} for elliptic systems with $p$-growth (see also \cite{Giaquinta_book,Stre80}) and by Giaquinta and Struwe \cite{GiaStru82} for parabolic systems with $2$-growth (i.e., $p=2$). But it was an open problem for about 20 years for parabolic problems with $p$-growth ($p \neq 2$), and then Kinnunen and Lewis obtained the result \cite{KiLewis00}.  We also refer to \cite{Parviainen08,Parviainen09-1,Parviainen09-2} for  global higher integrability for parabolic problems with $p$-growth and  \cite{BoSch12}, \cite{Bogelein08}, \cite{AnZhi05,BoDu11}, \cite{BoDuKi19} and \cite{BoDuKorSche19, GiaSchwar19} for higher integrability results for obstacle problems, higher order parabolic systems, parabolic systems with $p(x,t)$-growth, doubly nonlinear parabolic systems and porous medium systems, respectively.

Now let us state our higher integrability result for parabolic systems with Orlicz growth.  

\begin{theorem}\label{mainthm}
Let $\phi:[0,\infty)\to[0,\infty)$ be a weak $\Phi$-function satisfying \ainc{p} and \adec{q} 
with constant $L\geq 1$ and let $u$ be a 
local weak solution to \eqref{maineq} with structure conditions \eqref{A-condition}--\eqref{pq-range}.
There exists $\epsilon=\epsilon(n,N,p,q,L,\nu,\Lambda)>0$ such that $\phi(|Du|)\in L^{1+\epsilon}_{loc}(\Omega_I)$ with the following estimate: for any $Q_{4\rho}\Subset\Omega_I$,
\[
\fint_{Q_{\rho}} \phi(|Du|)^{1+\epsilon}\,dz \leq c \left[(\phi \circ \mathcal{D}^{-1})\bigg(\fint_{Q_{2\rho}}\phi(|Du|)\,dz\bigg)\right]^\epsilon\fint_{Q_{2\rho}}\phi(|Du|)\,dz
\]
for some $c=c(n,N,p,q,L,\nu,\Lambda)>0$, where 
\begin{equation}\label{D-def}
\mathcal{D}(t):=\min\{t^2,\phi(t)^{\frac{n+2}{2}}t^{-n}\}
\end{equation} 
and $\mathcal{D}^{-1}$ is the left-inverse of $\mathcal D$.
\end{theorem}

We remark that  when $\phi(t)=t^p$, we have  
$\mathcal{D}(t)=\min\{t^2,t^{\frac{p(n+2)}{2}-n}\}$
and so
\[
(\phi\circ \mathcal{D}^{-1})(t)
= \max\{t^{\frac{p}{2}},t^{\frac{2p}{p(n+2)-2n}}\}
\]
Therefore, our result exactly implies the known results for the $p$-growth case, see for instance \cite{BoDu11}.

We would like to introduce the novelties of our approach used in this paper. 
The main step is to obtain a reverse H\"older type inequality. 
In this step we cannot take advantage of the approach used in the $p$-growth case, 
which is why the higher integrability for parabolic problems with Orlicz growth has remained unsolved. 
The first issue is that techniques for $p>2$ and $p<2$ in the $p$-growth case are different
which is problematic in the Orlicz case. In this paper, we present a universal approach that is independent of whether the system \eqref{maineq} is degenerate, singular or neither.  The second problem is that the classical Gagliardo--Nirenberg interpolation inequality, which is an important ingredient in the $p$-growth case, is not applicable to the Orlicz setting. In order to overcome this problem, we derive an interpolation inequality for the Orlicz case, see Lemma~\ref{lem:GNOrlicz}.  The remaining part follows the approach used in \cite{KiLewis00} with modifications for the Orlicz setting
using recent tools from \cite{HarH_book}.

Our paper is organized as follow. In the next section, Section \ref{sec2}, we introduce notation, 
Orlicz functions and derive an interpolation inequality. In Section \ref{sec3}, we obtain 
a reverse H\"older inequality. Finally in the last section, Section \ref{sec4}, we prove 
the main result, Theorem~\ref{mainthm}.


\section{Preliminaries}\label{sec2}

\subsection{Notation}

For $w=(y,\tau)\in \mr^n\times \mr$ we denote the usual parabolic cylinder by  
\[
Q_r(w):=B_r(y)\times (\tau-r^2,\tau+r^2),
\]
where $B_r(y)$ is the open ball in $\R^n$ with center $y$ and radius $r$, and the intrinsic parabolic cylinder (with function $\phi$) by
\[
Q^{\lambda}_{r}(w):=B_r(y)\times \Lambda^\lambda_r(\tau)
\quad
\text{where }\ \Lambda^\lambda_r(\tau):= (\tau-\tfrac{r^2}{\phi_2(\lambda)},
\tau+\tfrac{r^2}{\phi_2(\lambda)}),
\]
and, for the function $\phi:[0,\infty)\to[0,\infty)$, we define
$$
\phi_1(t):=\frac{\phi(t)}{t}
\quad\text{and}\quad
\phi_2(t):=\frac{\phi(t)}{t^2}.
$$

Let $f,g:[0,\infty)\to [0,\infty)$. The function $f$ is said to be \textit{almost increasing}  if there exists $L\ge 1$ such that $f(t)\leq Lf(s)$  for all $0<t<s<\infty$. If $L=1$ we say $f$ is increasing.  
\textit{Almost decreasing} and decreasing are defined analogously. 
We say that $f$ and $g$ are equivalent, $f\approx g$ if there exists $L\geq 1$ such that $L^{-1}f(t)\leq g(t)\leq L f(t)$ for all $t\geq 0$. 

We define $(f)_{U} := \fint_U f\, dz := \frac{1}{|U|}\int_U f \,dz$.  


\subsection{Orlicz functions}\label{subsec2.2}
Let $\phi:[0,\infty)\to[0,\infty)$ and $p,q>0$. 
We introduce some conditions.
\begin{itemize}
\item[\normalfont(aInc)$_p$]\label{ainc} 
The map $(0,\infty)\ni t\mapsto \phi(t)/t^p$ is almost increasing with constant $L\geq 1$.
\vspace{0.2cm}
\item[\normalfont(aDec)$_q$]\label{adec} 
The map $(0,\infty)\ni t\mapsto \phi(t)/t^q$ is almost decreasing with constant $L\geq 1$.
\end{itemize}
Note that \ainc{p} implies \ainc{p'} for all $p'<p$ and \adec{q} implies \adec{q'} for all $q'>q$. 
If $\phi$ satisfies \ainc{p} and \adec{q}, then $p\leq q$ and for any $t\in(0,\infty)$ and $0<c <1 <C$,
\[
 c^q L^{-1} \phi(t) \leq  \phi(ct)\leq  c^pL \phi(t)
 \quad \text{and}\quad 
 C^pL^{-1} \phi(t) \leq  \phi(Ct)\leq  C^qL \phi(t).
\]
We shall use these inequalities numerous times later without explicit mention.

These conditions allow us to work easily with weak $\Phi$-functions, without resorting to 
tricks to ensure convexity.
\begin{definition} 
The function $\phi:[0,\infty)\to[0,\infty]$ is said to be a \emph{weak $\Phi$-function} 
if it is increasing with $\phi(0)=0$, $\lim_{t\to 0^+}\phi(t)=0$, $\lim_{t\to\infty}\phi(t)=\infty$ 
and it satisfies \ainc{1}.  
\end{definition}
As an example of the robustness of this definition, we note that $\sqrt{\phi(t^2)}$ need not be 
convex if $\phi$ is, but the \ainc{1} property is conserved.
Moreover, the condition \ainc{1} captures some essential features of convexity, as it allows us 
to use the following Jensen-type inequality (cf.\ Lemma~4.3.2, \cite{HarH_book}). 

\begin{lemma}[Jensen inequality] \label{lem:Jensen}
If $\phi:[0,\infty)\to [0,\infty]$ is increasing with $\phi(0)=0$ and satisfies 
\ainc{1} with constant $L\geq 1$, then 
\[
\phi\bigg( \frac{1}{L^2} \fint_U |f|\, dz\bigg) \le \fint_U \phi(|f|)\, dz.
\] 
\end{lemma}

Also, we can use the conditions effectively with Young-type inequalities 
and inverse functions (see the proof of Lemma~\ref{lem:1stTerm}). We recall the 
definition of the conjugate weak $\Phi$-function:
\[
\phi^*(s) := \sup_{t\ge 0} (st - \phi(t)).
\]
This definition directly implies Young's inequality:
\begin{equation}\label{young}
st\leq \phi(s)+\phi^*(t),\qquad s,t\geq 0.
\end{equation}
The exact value of $\phi^*$ can usually not be determined, but we have the following 
useful estimate which can be found in the proof of \cite[Theorem~2.4.10]{HarH_book}:
\begin{equation}\label{phi*phi}
\phi^*\big(\tfrac{\phi(t)}{t}\big) \approx \phi(t).
\end{equation}
This will be used multiple times in what follows.
Moreover, if $\phi$ is differentiable with $\phi'$ satisfying \adec{q}, 
then, by \cite[Lemma~3.6(2)]{HasOk},
\begin{equation}\label{eq:phi'}
\tfrac{\phi(t)}{t} \approx \phi'(t).
\end{equation}

\begin{remark}\label{rmkphiphic}
Suppose that $\phi$ is a weak $\Phi$-function which satisfies 
\ainc{p} and \adec{q} with $1\leq p\leq q$. We can define 
\[
\tphi(t):=\int_0^t s^{p-1}\sup_{\sigma\in [0,s]}\frac{\phi(\sigma)}{\sigma^p}\,ds. 
\]
Then $\tphi\approx \phi$ is differentiable and convex (since its derivative 
is increasing) and also satisfies \ainc{p} with $L=1$ and \adec{q}. Since $\phi(\sigma)>0$ 
when $\sigma>0$, $\tphi$ is also strictly increasing. Since the claims that 
we are proving are invariant under equivalence of weak $\Phi$-functions, we may thus 
assume when necessary that $\phi$ is differentiable, strictly increasing and a bijection.  
\end{remark}

We next introduce the left-inverse of a weak $\Phi$-function:
$$
\phi^{-1}(s):=\inf\{t\ge 0: \phi(t)\ge s\}.
$$ 
Clearly,  $\phi^{-1}(\phi(t))\le t$ and, if $\phi$ is continuous, $\phi(\phi^{-1}(s))\geq s$.  
Note that in view of  Remark~\ref{rmkphiphic}, we have $(\phi\circ \phi^{-1})(t)\approx (\phi^{-1}\circ \phi)(t)\approx t$ if $\phi$ satisfies  \adec{q} with $q\ge 1$.
By \cite[Proposition~2.3.7]{HarH_book}, $\phi^{-1}$ satisfies \ainc{\frac1q} or \adec{\frac1p} if and 
only if $\phi$ satisfies \adec{q} or \ainc{p}, respectively. 
From these facts and Lemma~\ref{lem:Jensen} we conclude the Jensen-inequality 
\begin{equation}\label{jensenineq-1}
\fint_U \phi(|f|)\, dz \le c \phi\bigg(\fint_U |f|\, dz \bigg)
\end{equation}
when $\phi$ satisfies \adec{1}.


Let us quote for later use a version of the standard iteration lemma which is particularly 
adapted to the Orlicz case \cite[Lemma~4.2]{HarHT17}. Recall that doubling means that 
$X(2s)\le C X(s)$, which is equivalent to \adec{q} for some $q<\infty$. 

\begin{lemma}\label{lem:iteration}
Let $Z$ be a bounded non-negative function in the  interval $[r,R] \subset \R$ 
and let $X$ be a doubling function in $[0, \infty)$.  
Assume that there exists $\theta\in [0,1)$ such that 
\[
Z(t) \leq X(\tfrac1{s-t}) + \theta Z(s)
\]
for all $r \leq t < s \leq R$. Then
\[
Z(r) \lesssim X(\tfrac1{R-r}),
\]
where the implicit constant depends only on the doubling constant and $\theta$.
\end{lemma}

We end this subsection with the following lemma and notation which will be used often.

\begin{lemma}\label{lem:1stTerm}
Assume that $\phi:[0,\infty)\to[0,\infty)$ is a weak $\Phi$-function satisfying \ainc{p} and \adec{q} with $1<p\leq q$ and that 
$$
\fint_{U}\phi(|Du|)\, dz\leq  \phi(\lambda).
$$
Then, for any $\theta_0\in [1-\frac1q,1]$ and  $\delta\in(0,1)$, 
\[
\begin{split}
A_0
\le 
\begin{cases}
\delta \lambda + c_\delta \phi^{-1}(\Theta_0) \\
c \lambda
\end{cases}
\end{split}
\]
for some $c_\delta=c(p,q,L,\delta)>0$, where
\begin{equation}\label{eq:A0T0}
A_0:= \frac{1}{\phi_2(\lambda)}\fint_{U}\phi_1 (|Du|)\,dz
\quad\text{and}\quad
\Theta_0 := \bigg(\fint_{U}\phi(|Du|)^{\theta_0}\,dz\bigg)^{\frac{1}{\theta_0}}.
\end{equation}
\end{lemma}
\begin{proof}
We write $f:=Du$. Since 
$\phi_1(\phi^{-1}(t)) \approx \frac t{\phi^{-1}(t)}$ 
satisfies \adec{1-\frac1q}, we find that $t\mapsto (\phi_1\circ \phi^{-1})(t^{1/\theta_0})$ 
satisfies \adec{1}. Then we obtain by \eqref{jensenineq-1} and H\"older's inequality that 
\[
\fint_{U} \phi_1(|f|)\,dz 
\le
c(\phi_1 \circ \phi^{-1})(\Theta_0)
\le
c (\phi_1 \circ \phi^{-1})\bigg(\fint_{U}\phi(|f|)\,dz\bigg) 
\le
c\phi_1(\lambda).
\]
We use only the first inequality to estimate $\frac1q$-part of the integral, and 
the whole inequality for the remaining $(1-\frac1q)$ of the integral. Thus  
\[
\frac{1}{\phi_2(\lambda)}\fint_{U}\phi_1 (|f|)\,dz
\le
\frac{c [(\phi_1\circ \phi^{-1})(\Theta_0) ]^{\frac{1}{q}}}{\phi_2(\lambda)} \bigg(\fint_{U}\phi_1 (|f|)\,dz\bigg)^{1-\frac{1}{q}}
\le
\frac{c\lambda}{\phi_1(\lambda)^{\frac{1}{q}}}[(\phi_1\circ \phi^{-1})(\Theta_0) ]^{\frac{1}{q}}.
\]

Let us define $H$ by 
\[
H(t):=G^{-1}(t^q), \quad \text{where }\ G(t):=\frac{t^{q+1}}{\phi(t)}.
\]
Observe that, since $G$ satisfies \ainc{1} and \adec{q+1-p},
$H$ satisfies \ainc{\frac{q}{q+1-p}} and \adec{q} and \[
H^{-1}(t) \approx \Big(\frac{t^{q+1}}{\phi(t)}\Big)^\frac1q =\frac{t}{\phi_1(t)^{1/q}}.
\]
From $H^{-1}(t)(H^*)^{-1}(t)\approx t$, we see that
$(H^{*})^{-1}(t)\approx \frac{t}{H^{-1}(t)}\approx\phi_1(t)^\frac{1}{q}$.
Then by Young's inequality we have that for any $\delta>0$,
\[
\frac{1}{\phi_2(\lambda)}\fint_{U}\phi_1 (|Du|)\,dz 
\le 
\delta H\Big(\frac{\lambda}{\phi_1(\lambda)^{1/q}}\Big) 
+ c_\delta H^{*}([(\phi_1\circ \phi^{-1})(\Theta_0) ]^{\frac{1}{q}})
\approx 
\delta \lambda +c_\delta  \phi^{-1}(\Theta_0). 
\]
For the upper bound $c\lambda$, we simply fix $\delta=1$ and use H\"older's inequality 
for $\Theta_0$ as before. 
\end{proof}

\subsection{Gagliardo--Nirenberg type interpolation inequality for Orlicz functions}

In this subsection, we obtain  a Gagliardo--Nirenberg type interpolation inequality 
involving Orlicz functions which will be a crucial ingredient in the proof of a reverse H\"older type estimate. 
Let us recall a usual scaling invariant version of the Gagliardo--Nirenberg interpolation inequality 
in balls (see \cite{Nirenberg}): 
for $\gamma>0$, $p\in [1,n)$, $q\geq 1$ and $\theta\in[0,1]$,
\begin{equation}\label{GNclassic}
\bigg(\fint_{B_{r}}\big|\tfrac{f}{r}\big|^\gamma\, dx\bigg)^{\frac1\gamma}
\leq 
c\bigg(\fint_{B_r}\left[|Df|^p+\big|\tfrac{f}{r}\big|^p\right]\,dx\bigg)^{\frac{\theta}{p}}\bigg(\fint_{B_r}\big|\tfrac{f}{r}\big|^q\,dx\bigg)^{\frac{1-\theta}{q}}
\end{equation}
for some $c=c(n,p)>0$, provided that
$$
\frac{1}{\gamma}\geq \frac\theta{p^*}+\frac{1-\theta}{q},
$$
where $p^*=\frac{np}{n-p}$.
Let us prove that this inequality also holds for $q\in (0,1)$. This may be known, but we have not found 
a proof (e.g.\ it is not mentioned in \cite{Tri14} which deals with extensions of the parameter 
ranges). 

\begin{lemma}[Gagliardo--Nirenberg inequality]\label{lem:GNineq}
Let $p\in [1,n)$ and $\gamma,q>0$. Then 
\[
\bigg(\fint_{B_{r}}\big|\tfrac{f}{r}\big|^\gamma\, dx\bigg)^{\frac1\gamma}
\leq c
\bigg(\fint_{B_r}\left[|Df|^p+\big|\tfrac{f}{r}\big|^p\right]\,dx\bigg)^{\frac{\theta}{p}}\bigg(\fint_{B_r}\big|\tfrac{f}{r}\big|^q\,dx\bigg)^{\frac{1-\theta}{q}}
\]
for some $c=c(n,p)>0$, provided that
\[
\frac{1}{\gamma}\geq \frac\theta{p^*}+\frac{1-\theta}{q}.
\]
\end{lemma}

\begin{proof}
For $q\ge 1$, the claim is just the Gagliardo--Nirenberg inequality \eqref{GNclassic} quoted above. 
For $q\in(0,1)$, choose $\tilde\theta\in(0,1)$ such that
\[
\tilde\theta p^* + (1-\tilde \theta)q = 1,
\]
Then by H\"older's inequality we have
\[
\fint_{B_r}\big|\tfrac{f}{r}\big|\,dx
\leq 
\bigg(\fint_{B_r}\big|\tfrac{f}{r}\big|^{p^*}\,dx\bigg)^{\tilde\theta} 
\bigg(\fint_{B_r}\big|\tfrac{f}{r}\big|^{q}\,dx\bigg)^{1-\tilde\theta}
\]
and by the Sobolev--Poincar\'e inequality
\[
\fint_{B_r}\big|\tfrac{f}{r}\big|^{p^*}\,dx
\le c 
\bigg(\fint_{B_r}\left[|Df|^p+\big|\tfrac{f}{r}\big|^p\right]\,dx\bigg)^\frac{p^*}p.
\]

We use the Gagliardo--Nirenberg inequality \eqref{GNclassic} 
for $(\gamma,p,q,\theta) = (\gamma,p,1,\theta_0)$
and the estimates from the previous paragraph to conclude that 
\begin{align*}
\bigg(\fint_{B_{r}}\big|\tfrac{f}{r}\big|^\gamma\, dx\bigg)^{\frac1\gamma}
&\le
c \bigg(\fint_{B_r}\left[|Df|^p+\big|\tfrac{f}{r}\big|^p\right]\,dx\bigg)^{\frac{\theta_0}{p}}
\bigg(\fint_{B_r}\big|\tfrac{f}{r}\big|\,dx\bigg)^{1-\theta_0} \\
&\le
c\bigg(\fint_{B_r}\left[|Df|^p+\big|\tfrac{f}{r}\big|^p\right]\,dx\bigg)^
{\frac{\theta_0+(1-\theta_0)\tilde\theta p^*}p}
\bigg(\fint_{B_r}\big|\tfrac{f}{r}\big|^{q}\,dx\bigg)^{(1-\tilde\theta)(1-\theta_0)}
\end{align*}
provided that
\[
\frac{1}{\gamma}\geq \frac{\theta_0}{p^*}+1-\theta_0.
\]

Let $\theta:= \theta_0+(1-\theta_0)\tilde\theta p^* = 1 - (1-\theta_0)(1-\tilde\theta)q$. 
Then the exponent of the first term 
on the right-hand side is $\frac\theta p$ and the exponent of the second term is 
$\frac{1-\theta}q$.
Thus we have the Gagliardo--Nirenberg inequality for parameter value $\theta$. 
It remains to check that the bound for $\gamma$ is correct. For this we calculate 
\[
\frac{\theta}{p^*}+\frac{1-\theta}q
=
\frac{\theta_0}{p^*} + \tilde\theta(1-\theta_0) + (1-\tilde\theta)(1-\theta_0)
=
\frac{\theta_0}{p^*}+1-\theta_0. \qedhere
\]
\end{proof}

We conclude this subsection by deriving a Gagliardo--Nirenberg type inequality for 
Orlicz functions.

\begin{lemma}\label{lem:GNOrlicz}
Suppose that $\psi:[0,\infty)\to[0,\infty)$ is a weak $\Phi$-function and satisfies 
\adec{q_1} for some $q_1\geq 1$. For $p\in [1,n)$ and $q_2>0$ we have
\[ 
\bigg(\fint_{B_{r}}\psi\big(\big|\tfrac{f}{r}\big|\big)^\gamma\, dx\bigg)^{\frac1\gamma}
\le c \bigg(\fint_{B_r}\left[\psi(|Df|)^p+\psi\big(\big|\tfrac{f}{r}\big|\big)^p\right]\,dx\bigg)^{\frac{\theta}{p}}
\psi\bigg(\Big(\fint_{B_r}\big|\tfrac{f}{r}\big|^{q_2}\,dx\Big)^\frac1{q_2}\bigg)^{1-\theta}
\]
for some $c=c(n,L,q_1,q_2)>0$, provided that 
\[
\frac{1}{\gamma}\geq \frac\theta{p^*}+\frac{(1-\theta)q_1}{q_2}.
\]
 \end{lemma}
 
\begin{proof}
In view of Remark~\ref{rmkphiphic}, if suffices to consider $\psi$ that is strictly increasing and differentiable with $\psi'(t)\ge 0$. 
Then using Young's inequality \eqref{young}, \eqref{eq:phi'} and \eqref{phi*phi} we have that   
\[
|D(r \psi(|\tfrac fr|))|=\psi'(|\tfrac fr|) |Df| \le  \psi(|Df|) + \psi^*(\psi'(|\tfrac fr|)) \le \psi(|Df|) + c \psi (|\tfrac fr|).
\]
Hence applying Lemma~\ref{lem:GNineq} with $q=\frac{q_2}{q_1}$ to the function $r \psi(\left|\frac fr\right|)$, we conclude that
\[\begin{split}
\bigg(\fint_{B_{r}}\psi\big(\big|\tfrac{f}{r}\big|\big)^\gamma\, dx\bigg)^{\frac1\gamma}
&\le c
\bigg(\fint_{B_r}\left[\psi(|Df|)^p+\psi\big(\big|\tfrac{f}{r}\big|\big)^p\right]\,dx\bigg)^{\frac{\theta}{p}}
\bigg(\fint_{B_r}\psi\big(\big|\tfrac{f}{r}\big|\big)^{q}\,dx\bigg)^\frac{1-\theta}{q} \\
& \le c
\bigg(\fint_{B_r}\left[\psi(|Df|)^p+\psi\big(\big|\tfrac{f}{r}\big|\big)^p\right]\,dx\bigg)^{\frac{\theta}{p}}
\psi\bigg(\Big(\fint_{B_r}\big|\tfrac{f}{r}\big|^{q_2}\,dx\Big)^\frac1{q_2}\bigg)^{1-\theta};
\end{split}
\]
here the last inequality follows from Lemma~\ref{lem:Jensen} since $t\mapsto \psi^{-1}(t^{1/q})^{q_2}$ 
satisfies \ainc{1}.
\end{proof}

\section{Poincare and reverse H\"older type inequalities}
\label{sec3}

In this section, we derive a reverse H\"older type inequality for the gradients 
of weak solutions to \eqref{maineq} on regions satisfying a balancing condition, \eqref{lemreverseass}. We suppose that $\phi:[0,\infty)\to[0,\infty)$ is a weak $\Phi$-function and satisfies \ainc{p} and \adec{q} with constant $L\geq 1$ and that 
$p$ and $q$ satisfy \eqref{pq-range}. We start with a Caccioppoli type inequality. 
  
\begin{lemma}[Caccioppoli inequality]\label{lem:caccio}
Let $u$ be a weak solution to \eqref{maineq} with \eqref{A-condition} and 
$Q^\lambda_R\Subset\Omega_I$ with $\lambda,R>0$. For $r\in[\frac{R}{4},R)$ 
and $a=(a_1,\dots, a_N)\in\mr^N$, we have 
\[
\begin{split}
\phi_2(\lambda) \sup_{s\in\Lambda^\lambda_r}\fint_{B_r}  \bigg|\frac{u(s)-a}{r}\bigg|^2\,dx &+ \fint_{Q_r^\lambda} \phi(|Du|)\, dz \\
&\hspace{-2cm}\leq c\fint_{Q^\lambda_R}\left[\phi_2(\lambda)\Big|\frac{u-a}{R-r}\Big|^2+ \phi\Big(\Big|\frac{u-a}{R-r}\Big|\Big)\right]\, dz
\end{split}
\]
for some $c=c(n,N,p,q,L,\nu,\Lambda)>0$, where $u(s)=u(x,s)$. 
\end{lemma}

\begin{proof}
We assume without loss of generality that $Q_R^\lambda$ is centered at the origin. 
Let $\eta\in C^1_0(B_R)$ with $\eta\equiv 1$ in $B_r$  and $|D\eta|\leq c
/(R-r)$ and $\tau\in C^1(\mr)$ with $\tau\equiv 0$ in $(-\infty,-R^2/\phi_2(\lambda)]$, $\tau\equiv 1$ in $[-r^2/\phi_2(\lambda),\infty)$ and $0\leq \tau'\leq \phi_2(\lambda)/(R-r)^2$.
Define
\[
\zeta(x,t):=\eta(x)^{q}\tau(t)^2 (u(x,t)-a)
\]
and $\sigma := -R^2/\phi_2(\lambda)$.
Using $\zeta$ as a test function, we have that for $s\in \Lambda^{\lambda}_{R}$,
\begin{equation}\label{weakform1}
\int_\sigma^s\int_{B_R}\left[\partial_t u_i  \zeta_i +A_i(x,t,u,Du) \cdot D\zeta_i\right]\, dx\,dt=0,\quad i=1,\dots,N,
\end{equation}
see Remark~\ref{rmkSteklov}, below. Since $\partial_t a=0$, we note that
\[\begin{split}
\int_\sigma^s\int_{B_R}\partial_t u_i  \zeta_i \, dx\,dt
&= \int_\sigma^s\int_{B_R}\frac12\partial_t [\eta^{q}\tau^2 (u_i-a_i)^2] - \eta_i^{q}\tau\tau' (u_i-a_i)^2 \, dx\,dt \\
&= \frac12 \int_{B_R} \eta^{q}\tau(s)^2 (u_i(s)-a_i)^2\,dx - \int_\sigma^s\int_{B_R} \eta^{q}\tau\tau' (u_i-a_i)^2 \, dx\,dt,
\end{split}\]
where $u_i(s)=u_i(x,s)$. 
Then, summing \eqref{weakform1} for $i=1,\dots,N$ and using \eqref{A-condition}, we have 
\[\begin{split}
&\frac12 \int_{B_R} \eta^{q}\tau(s)^2 |u(s)-a|^2\,dx + \nu\int_\sigma^s\int_{B_R} \eta^{q}\tau^2 \phi(|Du|)\, dx\,dt\\
& \leq   -q\int_\sigma^s\int_{B_R}\eta^{q-1}\tau^2
\sum_{i=1}^N(u_i-a_i) A_i(x,u,Du) \cdot D\eta\, dx\,dt 
+\int_\sigma^s\int_{B_R} \eta^{q}\tau\tau' |u-a|^2 \, dx\,dt\\
& \leq  c \int_\sigma^s\int_{B_R}\tau^2 \eta^{q-1}
\frac{\phi(|Du|)}{|Du|}\Big|\frac{u-a}{R-r}\Big| \, dx\,dt 
+\int_\sigma^s\int_{B_R}  \phi_2(\lambda)\Big|\frac{u-a}{R-r}\Big|^2\, dx\,dt.
\end{split}
\]
Moreover, by Young's inequality \eqref{young} and \eqref{phi*phi} and 
since $\phi^*$ satisfies \ainc{\frac{q}{q-1}} \cite[Proposition~2.4.13]{HarH_book}, we have that for any $\epsilon>0$ 
\[\begin{split}
\eta^{q-1}\frac{\phi(|Du|)}{|Du|}\Big|\frac{u-a}{R-r}\Big|
&\leq \epsilon \phi^*\left(\eta^{q-1}
\frac{\phi(|Du|)}{|Du|}\right) +c_\epsilon\phi\Big(\Big|\frac{u-a}{R-r}\Big|\Big)\\
&\leq c\epsilon \eta^q \phi (
|Du|) +c_\epsilon\phi\Big(\Big|\frac{u-a}{R-r}\Big|\Big).
\end{split}\]
We choose $\epsilon$ so small that the first term can be absorbed in the left-hand side. 
Therefore, combining the above inequalities and using the fact that  $\tau\equiv 1$ in 
$\Lambda^\lambda_r$ and $\eta\equiv 1$ in $B_r$, we have 
\[\begin{split}
\int_{B_r}  |u(s)-a|^2\,dx + &\int_{-\frac{r^2}{\phi_2(\lambda)}}^s\int_{B_r} \phi(|Du|)\, dx\,dt 
\leq c\int_{Q^\lambda_R}\left[\phi_2(\lambda)\Big|\frac{u-a}{R-r}\Big|^2+ \phi\Big(\Big|\frac{u-a}{R-r}\Big|\Big)\right]\, dz
\end{split}\]
for all $s\in \Lambda^{\lambda}_r$. Finally, we obtain the claim by dividing both 
sides by $r^{n+2}\phi_2(\lambda)$ and taking into consideration 
that $|B_r|\approx r^n$ while $|Q^\lambda_R| \approx r^{n+2}\phi_2(\lambda)$.  
\end{proof}

\begin{remark}\label{rmkSteklov}
When we consider parabolic problems and want to obtain useful estimates such as Caccioppoli inequalities,  we have to use test functions depending on the weak solution $u$ in the weak formulation \eqref{weakform}. However, the weak solution $u$ to the parabolic system \eqref{maineq} may not be differentiable in the time variable. In fact, we do not need differentiability with respect to the time variable when we prove the existence of weak solution to parabolic problems. In order overcome this difficulty, one way is to consider Steklov averages, see \cite{DiBenedetto_book} and also \cite{BaLin17}. However, since this argument is now quite standard, we shall abuse the notation 
$\partial_t u$ without further explanation.
\end{remark}

Let $\eta\in C_0^\infty(B_{\rho})$  be a cut-off function such that $0\leq \eta \leq 1$, $\eta\equiv 1$ in $B_{\rho/2}$, $|D\eta|\leq \frac 4\rho$.  We note that $\|\eta\|_{1}\approx |B_\rho|$. Define
$$
(u)_{\rho}^\lambda:= \frac{1}{\|\eta\|_1} \fint_{\Lambda_\rho^\lambda}\int_{B_\rho} u \eta \, dx\,dt
\quad\text{and}\quad
\langle u \rangle_\eta(s):= \frac{1}{\|\eta\|_1}\int_{B_\rho}u(x,s)\eta \, dx \ \ \text{for }s\in \Lambda^\lambda_\rho.
$$
Let us start with a complicated ``Sobolev--Poincar\'e'' inequality. 

\begin{lemma} \label{lem:poin1}
Let $u$ be a weak solution to \eqref{maineq} with \eqref{A-condition} and $Q^\lambda_{4\rho}\Subset\Omega_I$ with $\lambda>0$ and $\rho\le r<R\le  4\rho$.
For a weak $\Phi$-function $\psi$ satisfying \ainc{p_1} and \adec{q_1}, $1\le p_1\leq q_1$, we have
\[
\begin{split}
\fint_{Q^\lambda_{r}}\psi\bigg(\bigg|\frac{u-(u)^\lambda_\rho}{r}\bigg|\bigg)\,dz
&\le
c \psi(A_0) + 
c \psi\big(T(r,R)^\frac12\big)^{(1-\theta_0)} \fint_{Q_r^\lambda}\psi(|Du|)^{\theta_0} \,dz, 
\end{split}
\]
for some $c=c(n,N,p,q,p_1,q_1,\theta_0,L,\nu,\Lambda)>0$ provided that 
\[
\theta_0 p_1\in[1,n)
\quad\text{and}\quad
\frac{ n q_1}{n q_1 + 2p_1}\le \theta_0 \le 1.
\]
Here $A_0$ is from \eqref{eq:A0T0} with $U:=Q^\lambda_r$ and 
\begin{align*}
T(r,R)
:= 
\fint_{Q_{R}^\lambda}\bigg[ \bigg|\frac{u-(u)_\rho^\lambda}{R-r}\bigg|^2 + 
\frac1{\phi_2(\lambda)}\phi\bigg(\bigg|\frac{u-(u)_\rho^\lambda}{R-r}\bigg|\bigg)\bigg]\, dz + 
A_0^2,
\end{align*}
\end{lemma}
\begin{proof}
By the triangle inequality, 
\begin{equation}\label{lem:poin1:pf1}\begin{split}
\fint_{Q_r^\lambda}\psi\bigg(\bigg|\frac{u-(u)_\rho^\lambda}{r}\bigg|\bigg) \, dz
&=
\fint_{Q_r^\lambda}\psi\bigg(\bigg|\frac{u(z)-\langle u\rangle_\eta(t)
+\langle u\rangle_\eta(t)-(u)_\rho^\lambda}{r}\bigg|\bigg) \,dz \\
&\hspace{-2cm}\le
c\fint_{\Lambda_r^\lambda} \psi\bigg(\bigg| \frac{\langle u\rangle_{\eta}(t)-(u)_\rho^\lambda}r \bigg|\bigg) \,dt
+
c\fint_{Q_r^\lambda}\psi\bigg(\bigg|\frac{u(z)-\langle u\rangle_{\eta}(t)}{r}\bigg|\bigg) \,dz.  
\end{split}\end{equation}
We first take care of the first term. 
By the definition of $\langle u\rangle_\eta$ and using the weak formulation \eqref{weakform} with test-function $\zeta(x,t):=\eta(x)$, we find that for each $i=1,\dots, N$ and $t\in \Lambda_r^\lambda$,
\begin{equation}\label{uaverageueta}
\begin{split}
|\langle u_i\rangle_\eta(t) - (u_i)_\rho^\lambda|
&\le
\sup_{\sigma\in  \Lambda_r^\lambda } |\langle u_i\rangle_\eta(t)-\langle u_i\rangle_\eta(\sigma)| =
 \sup_{\sigma\in  \Lambda_r^\lambda }\bigg|\int^{t}_{\sigma} \partial_t \langle u_i\rangle_\eta(s)\, ds\bigg|  \\
&
= \sup_{\sigma\in  \Lambda_r^\lambda }\bigg|\int^{t}_{\sigma}\frac{1}{\|\eta\|_1}\int_{B_r}\partial_tu_i(x,s)\eta(x) \, dx\, ds\bigg|\\
&\approx \sup_{\sigma\in  \Lambda_r^\lambda } \bigg|\int^{t}_{\tau}\fint_{B_r} A_i(x,s,u,Du)\cdot D\eta\,dx\,ds\bigg|\\
&\leq \frac{cr}{\phi_2(\lambda)}\fint_{Q_r^\lambda} \phi_1 (|Du|)\,dz = r A_0.
\end{split}
\end{equation}
This gives the first term on the right-hand side of the claim.

We next use the Gagliardo--Nirenberg inequality (Lemma~\ref{lem:GNOrlicz}) with 
$(\psi,\gamma,p,q_1,q_2)$ given by $(\psi^{1/p_1}, p_1,\theta_0 p_1,\frac {q_1}{p_1},2)$
to conclude that  
\[
\fint_{B_r}\psi\big(\big|\tfrac{f}{r}\big|\big)\, dx
\le
c\fint_{B_r}\Big[\psi(|Df|)^{\theta_1}+\psi\big(\big|\tfrac{f}{r}\big|\big)^{\theta_0}\Big]\,dx
\ \psi\Bigg(\bigg[\fint_{B_r}\big|\tfrac{f}{r}\big|^2\,dx\bigg]^\frac12\Bigg)^{1-\theta_0}
\]
provided $\theta_0 p_1\in[1,n)$ and
\[
\frac1{p_1} \ge 
\frac{\theta_0}{(\theta_0 p_1)^*} + \frac{1-\theta_0}2 \frac{q_1}{p_1}
=
\frac{1}{p_1} - \frac{\theta_0}{n} + \frac{1-\theta_0}2 \frac{q_1}{p_1}.
\]
This can be written as $\theta_0\ge \frac{n q_1}{n q_1 + 2p_1}$.

The previous inequality for $f:=u-\langle u\rangle_\eta$ on each time slice gives 
\begin{equation}\label{lem:poin1:pf2}\begin{split}
\fint_{Q_r^\lambda}\psi\Big(\Big|\frac{u(z)-\langle u \rangle_\eta(t)}{r}\Big|\Big) \,dz
& = 
\fint_{\Lambda_r^\lambda} \fint_{B_r}\psi\big(\big|\tfrac{f(x,t)}{r}\big|\big) \,dx\,dt \\
&\hspace{-2cm} \le c
\fint_{Q_r^\lambda}[\psi(|Df|)^{\theta_0} + \psi(\tfrac fr)^{\theta_0}]\,dz \,
\psi\Bigg(\bigg( \sup_{t\in \Lambda_r^\lambda} \fint_{B_r}
\big|\tfrac{f(x,t)}{r}\big|^2\,dx\bigg)^\frac12\Bigg)^{1-\theta_0}.
\end{split}\end{equation}
In the first term we then use the following weighted Poincar\'e inequality for each time slice:
\[\begin{split}
\fint_{B_r}\psi\big(\big|\tfrac{f(x,t)}{r}\big|\big)^{\theta_0} \,dx
&=\fint_{B_r}\psi\Big(\Big|\frac{u(x,t)-\langle u\rangle_{\eta}(t)}{r}\Big|\Big)^{\theta_0} \,dx \\
& \leq c \fint_{B_r}\psi(|Du(x,t)|)^{\theta_0} \,dx
=c \fint_{B_r}\psi(|Df(x,t)|)^{\theta_0} \,dx,
\end{split}\] 
see \cite[Lemma~6.2.5 and Corollary~7.4.1(b)]{HarH_book} (here we need $\theta_0 p_1\ge 1$). 
Finally, from the Caccioppoli inequality (Lemma~\ref{lem:caccio}) 
and \eqref{uaverageueta} we conclude that 
\begin{align*}
\sup_{t\in \Lambda_r^\lambda} \fint_{B_r}\big|\tfrac{f(x,t)}{r}\big|^2 \,dx
&\le c \sup_{t\in \Lambda_r^\lambda} \fint_{B_r}\bigg|\frac{u(x,t)-(u)_\rho^\lambda}{r}\bigg|^2\,dx
+ c \sup_{t\in \Lambda_r^\lambda} \fint_{B_r}\bigg|\frac{(u)_\rho^\lambda-\langle u\rangle_{\eta}(t)}{r}\bigg|^2\,dx\\
&\le c
\fint_{Q_{R}^\lambda}\bigg[ \bigg|\frac{u-(u)_\rho^\lambda}{R-r}\bigg|^2 + 
\frac1{\phi_2(\lambda)}\phi\bigg(\bigg|\frac{u-(u)_\rho^\lambda}{R-r}\bigg|\bigg)\bigg]\, dz + 
cA_0^2. 
\end{align*}
Combining \eqref{lem:poin1:pf1}, \eqref{uaverageueta} and  \eqref{lem:poin1:pf2}, we obtain 
the claim.
\end{proof}

If we choose $\theta_0=p_1=1$ in the previous lemma, we obtain the following result, 
since the complicated term involving $T$ vanishes as its exponent is zero. 

\begin{corollary}[Poincar\'e inequality] \label{cor:poin1a}
Let $u$ be a weak solution to \eqref{maineq} with \eqref{A-condition} and $Q^\lambda_r\Subset\Omega_I$ with $\lambda>0$. 
For a weak $\Phi$-function $\psi$ satisfying \adec{q_1}  we have
\[
\fint_{Q^\lambda_{r}}\psi\bigg(\bigg|\frac{u-(u)^\lambda_\rho}{r}\bigg|\bigg)\,dz
\le
c\fint_{Q^\lambda_r}\psi(|Du|)\,dz + c \psi (A_0).
\]
\end{corollary}

Over the course of the next two results we will show how the extra terms in 
the previous lemma can be estimated by suitable quantities when we are in suitable 
intrinsic cylinders.

\begin{lemma}\label{lem:2ndTerm}
We assume the assumptions of Lemma~\ref{lem:poin1}, and additionally that 
$$
\fint_{Q^\lambda_{4\rho}}\phi(|Du|)\, dz\leq  \phi(\lambda).
$$
Then,
for some $c=c(n,N,p,q,p_1,q_1,\theta_0,L,\nu,\Lambda)>0$, 
\[
\begin{split}
\fint_{Q^\lambda_{2\rho}}\psi\bigg(\bigg|\frac{u-(u)^\lambda_\rho}{\rho}\bigg|\bigg)\,dz
&\le
c \psi (A_0)
+
c \psi(\lambda)^{1-\theta_0} \fint_{Q_{2\rho}^\lambda}\psi(|Du|)^{\theta_0} \,dz,
\end{split}
\]
where $A_0$ is from \eqref{eq:A0T0} with $U:=Q^\lambda_{2\rho}$
\end{lemma}
\begin{proof} 
The claim follows once we show that $T$ from Lemma~\ref{lem:poin1} satisfies 
$T(2\rho,3\rho)\leq c\lambda$. 
We first note from Lemma~\ref{lem:1stTerm} with $U=Q^\lambda_r$ that
\begin{equation}\label{lem:1stTerm-1}
\frac{1}{\phi_2(\lambda)}\fint_{Q^\lambda_r}\phi_1 (|Du|)\,dz\leq c\phi(\lambda) \quad \text{for all }\ r\in[\rho,4\rho].
\end{equation}
Using this and Corollary~\ref{cor:poin1a} with $\psi:=\phi$, we find that for any $\rho\leq r<R\leq 4\rho$,
\[
\fint_{Q^\lambda_{R}}\phi\bigg(\bigg|\frac{u-(u)^\lambda_\rho}{R}\bigg|\bigg)\,dz
\le
c\fint_{Q^\lambda_R}\phi(|Du|)\,dz 
+ c \phi\bigg(\frac{1}{\phi_2(\lambda)}\fint_{Q^\lambda_R}\phi_1 (|Du|)\,dz\bigg) 
\le
c \phi(\lambda) 
\]
and hence
\[
\frac1{\phi_2(\lambda)}\fint_{Q^\lambda_{R}}\phi\bigg(\bigg|\frac{u-(u)^\lambda_\rho}{R-r}\bigg|\bigg)\,dz
\le
\big(\tfrac{R}{R-r}\big)^{q}
\frac1{\phi_2(\lambda)}\fint_{Q^\lambda_{R}}\phi\bigg(\bigg|\frac{u-(u)^\lambda_\rho}{R}\bigg|\bigg)\,dz
\le
c\big(\tfrac{R}{R-r}\big)^{q} \lambda^2. 
\]
With the previous inequalities, 
we have the following estimate for $T$ from Lemma~\ref{lem:poin1}: for any $\rho\leq r <R \leq 4\rho$,
\begin{equation}\label{eq:T}
T(r,R) 
\le
\big(\tfrac{R}{R-r}\big)^{2} \fint_{Q_{R}^\lambda} \bigg|\frac{u-(u)_\rho^\lambda}{R}\bigg|^2 \, dz + 
c \big(\tfrac{R}{R-r}\big)^{q} \lambda^2.
\end{equation}

Let $p_0:=\frac{2n}{n+2}$, $\psi(t):=t^2$ and $\theta_0:=\frac{p_0}2$. 
Then by Lemma~\ref{lem:Jensen} for the map $t\mapsto\phi(t^{1/p_0})$,
\[
\fint_{Q^\lambda_{r}}\psi(|Du|)^{\theta_0}\,dz
=
\fint_{Q^\lambda_{r}}|Du|^{p_0}\,dz
\le
c \phi^{-1}\bigg(\fint_{Q^\lambda_{r}}\phi(|Du|)\,dz\bigg)^{p_0}
\le
c\lambda^{p_0}
\]
for all $r\in[\rho,4\rho]$. With the last two estimates, \eqref{lem:1stTerm-1} and Young's inequality, 
Lemma~\ref{lem:poin1} gives us in this case that
\[\begin{split}
\fint_{Q^\lambda_r}\bigg|\frac{u-(u)^\lambda_\rho}{r}\bigg|^2\,dz
&\le
c \lambda^2
+c\lambda^{p_0}\bigg(\big(\tfrac{R}{R-r}\big)^{2}\fint_{Q_{R}^\lambda} \bigg|\frac{u-(u)_\rho^\lambda}{R}\bigg|^2 \, dz + 
\big(\tfrac{R}{R-r}\big)^{q} \lambda^2 \bigg)^{1-\theta_0}\\
&\le
c \big(\tfrac{R}{R-r}\big)^{\frac{2q}{n+2}} \lambda^{2}
+c\lambda^{p_0}\big(\tfrac{R}{R-r}\big)^{\frac{4}{n+2}}\bigg(\fint_{Q_{R}^\lambda} \bigg|\frac{u-(u)_\rho^\lambda}{R}\bigg|^2 \, dz\bigg)^\frac{2}{n+2}\\
&\le
\epsilon \fint_{Q_{R}^\lambda} \bigg|\frac{u-(u)_\rho^\lambda}{R}\bigg|^2 \, dz 
+  c \big(\tfrac{R}{R-r}\big)^{\frac{2q}{n+2}} \lambda^2
+ c_\epsilon \big(\tfrac{R}{R-r}\big)^{\frac{4}{n}} \lambda^{2}.
\end{split}\]
for any $\epsilon\in(0,1)$. Then, since $\rho\leq r<R\leq 4\rho$, we have
\[
\underbrace{\fint_{Q^\lambda_r}\bigg|\frac{u-(u)^\lambda_\rho}{\rho}\bigg|^2\,dz}_{=Z(r)}
\le 
\frac{1}{2}
\underbrace{\fint_{Q^\lambda_R}\bigg|\frac{u-(u)^\lambda_\rho}{\rho}\bigg|^2\,dz}_{=Z(R)}
 +
\underbrace{c \Big[\big(\tfrac{4\rho}{R-r}\big)^{\frac{2q}{n+2}} + \big(\tfrac{4\rho}{R-r}\big)^{\frac{4}{n}} \Big]\lambda^{2}}_{=X(1/(R-r))}.
\]
Therefore, by applying a standard iteration lemma~\ref{lem:iteration}, we 
obtain $Z(r) \le c X(1/(R-r))$, i.e.\ 
\[
\fint_{Q^\lambda_r}\bigg|\frac{u-(u)^\lambda_\rho}{\rho}\bigg|^2\,dz 
\le c \Big[\big(\tfrac{\rho}{R-r}\big)^{\frac{2q}{n+2}} + \big(\tfrac{\rho}{R-r}\big)^{\frac{4}{n}} \Big]\lambda^{2}
\]
for any $\rho\leq r<R\leq 4\rho$. 
We choose here $(r,R)=(3\rho,4\rho)$ and use it in \eqref{eq:T} with $(r,R)=(2\rho,3\rho)$ to conclude $T(2\rho,3\rho)\le c \lambda^2$.
\end{proof}

Now, we derive a reverse H\"older inequality.

\begin{lemma}\label{lem:reverse}
Let $u$ be a weak solution to \eqref{maineq} with \eqref{A-condition} and 
$Q^\lambda_{4\rho}\Subset\Omega_I$ with $\lambda,\rho>0$. Suppose that 
\begin{equation}\label{lemreverseass}
\phi(\lambda) \leq \fint_{Q_\rho^\lambda}\phi(|Du|)\, dz
\quad\text{and}\quad
\fint_{Q_{4\rho}^\lambda}\phi(|Du|)\, dz\leq \phi(\lambda).
\end{equation}
Then there exist $\theta=\theta(n,p,q)\in(0,1)$ and $c=c(n,N,p,q,L,\nu,\Lambda)>0$ such that
\[
\fint_{Q^\lambda_\rho}
\phi(|Du|)\,dz\leq c\bigg(\fint_{Q^\lambda_{4\rho}}\phi(|Du|)^{\theta}\,dz\bigg)^{\frac{1}{\theta}}.
\]
\end{lemma}
\begin{proof} 
We denote $p_0 := \frac{2n}{n+2}$, and $A_0$ and $\Theta_0$ as in \eqref{eq:A0T0}
with $U:=Q^\lambda_{2\rho}$.
By the Caccioppoli inequality (Lemma~\ref{lem:caccio}) with $a:=(u)_\rho^\lambda$, 
we find that 
\begin{equation}\label{22}
\fint_{Q_\rho^\lambda} \phi(|Du|)\, dz 
\le
c \phi_2(\lambda) \fint_{Q^\lambda_{2\rho}}\bigg|\frac{u-(u)_{\rho}^\lambda}{\rho}\bigg|^2\,dz 
+ c \fint_{Q^\lambda_{2\rho}} \phi\bigg(\bigg|\frac{u-(u)_{\rho}^\lambda}{\rho}\bigg|\bigg)\, dz .
\end{equation}
We then estimate the last two integrals. 

By Lemma~\ref{lem:2ndTerm} for $\psi:=\phi$, considering also Lemma~\ref{lem:1stTerm} 
with $U=Q^\lambda_{2\rho}$
 and Young's inequality, we have that for any $\delta\in(0,1)$
\begin{equation}\label{22-1}
\fint_{Q^\lambda_{2\rho}}\phi\bigg(\bigg|\frac{u-(u)^\lambda_\rho}{\rho}\bigg|\bigg)\,dz
\le
c \phi(A_0) + c \phi(\lambda)^{(1-\theta_0)} \Theta_0^{\theta_0}
\le 
c_{\delta} \Theta_0 + c\delta \phi(\lambda). 
\end{equation}
Using the same steps in the case $\psi(t):=t^2$ and $\theta_0:=\frac{p_0}2$, 
we conclude that for any $\delta\in(0,1)$
\[\begin{split}
\bigg(\fint_{Q^\lambda_{2\rho}}\bigg|\frac{u-(u)^\lambda_\rho}{\rho}\bigg|^2\,dz\bigg)^{\frac{1}{2}}
&\le
c A_0
+ c \bigg(\lambda^{2-p_0} \fint_{Q^\lambda_{2\rho}}|Du|^{p_0} \,dz\bigg)^{\frac{1}{2}}\\
&\le c_\delta \bigg(\fint_{Q^\lambda_{2\rho}}|Du|^{p_0}\,dz\bigg)^\frac1{p_0} +c A_0 + \delta \lambda.
\end{split}\]
In particular, we also have 
\[
\bigg(\fint_{Q^\lambda_{2\rho}}\bigg|\frac{u-(u)^\lambda_\rho}{\rho}\bigg|^2\,dz\bigg)^{\frac{1}{2}}
\leq c \lambda.
\]
Multiplying the previous two inequalities, and using Young's inequality with 
\eqref{phi*phi} for the second step and  Lemma~\ref{lem:1stTerm}, 
we obtain that  for any $\delta\in (0,1)$
\begin{equation}\label{22-2}
\begin{split}
\phi_2(\lambda) \fint_{Q^\lambda_{2\rho}}\bigg|\frac{u-(u)^\lambda_\rho}{\rho}\bigg|^2\,dz
&\le
c\phi_1(\lambda) \bigg[c_\delta\bigg(\fint_{Q^\lambda_{2\rho}}|Du|^{p_0}\,dz\bigg)^\frac1{p_0} + A_0 + \delta \lambda \bigg] \\
&\le
c_\delta \phi\Bigg(\bigg(\fint_{Q^\lambda_{2\rho}}|Du|^{p_0}\,dz\bigg)^\frac1{p_0}\Bigg) + c_\delta\phi(A_0) + c \delta \phi(\lambda) \\
&\le
c_{\delta}\Theta_0 +c\delta \phi(\lambda),
\end{split}
\end{equation}
where the last step follows from Jensen's inequality (Lemma~\ref{lem:Jensen}) 
when $\theta_0\ge \frac{p_0}p$ 
so that $\phi(t^{1/p_0})^{\theta_0}$ satisfies \ainc{1}.

Finally, inserting \eqref{22-1} and \eqref{22-2} into \eqref{22}, we find that 
\[
\fint_{Q_\rho^\lambda} \phi(|Du|)\, dz \le c_\delta \Theta_0 + c \delta \phi(\lambda) .
\]
Choosing $\delta$ so small that $c\delta=\frac12$ and 
absorbing the term in the left-hand side by \eqref{lemreverseass} we have the reverse 
H\"older inequality.
\end{proof}


\section{Proof of higher integrability}\label{sec4}

Now we prove the main result, Theorem~\ref{mainthm}. 

\textit{Step 1.}
Let $\phi:[0,\infty)\to [0,\infty)$ 
be a weak $\Phi$-function and satisfy \ainc{p} and \adec{q} with \eqref{pq-range}. 
In view of Remark~\ref{rmkphiphic}, we can assume without loss 
of generality that $\phi$ is differentiable, strictly increasing and satisfies \ainc{p} with $L=1$. 
We also recall  
\begin{equation}\label{Dineq1}
\mathcal{D}(t):=\min\{t^2,\phi(t)^{\frac{n}{2}+1}t^{-n}\}=\min\{1,\phi_2(t)^{\frac{n+2}{2}}\}t^2
\end{equation}
from \eqref{D-def}. Then $\mathcal D$ is increasing and 
from \ainc{p} of $\phi$ we have 
\begin{equation} \label{Dineq2}
C^{\min\{2,\frac{p(n+2)-2n}{2}\}} \mathcal{D}(t)
=
\min\Big\{C^2, C^{\frac{p(n+2)-2n}{2}}\Big\} \mathcal{D}(t)
\le 
\mathcal{D}(Ct) 
\end{equation}
for all $t>0$ and $C\ge 1$.  

\textit{Step 2.}
Fix $Q_{2\rho}\Subset \Omega_I$. We define
\begin{equation}\label{lambda0}
\lambda_0:=\mathcal{D}^{-1}\bigg(\fint_{Q_{2\rho}}\phi(|Du|)\,dz\bigg)
\end{equation}
and, for $0<s\leq 2$ and $\lambda>0$, 
\[
E(s,\lambda):=\{z\in Q_{s\rho}:|Du(z)|>\lambda\}.
\]
We  next fix any $1\leq s_1<s_2\leq 2$ and any $\lambda$ satisfying   
\begin{equation}\label{lambda1}
\lambda \geq \lambda_1:= \Big(\frac{40}{s_2-s_1}\Big)^
{(n+2)\max\{\frac{1}{2}, \frac2{p(n+2)-2n}\} }\lambda_0.
\end{equation}
With this $\lambda$ we also define
\begin{equation}\label{rlambda}
r_\lambda:= 
\min\{1,\phi_2(\lambda)^{\frac12}\}(s_2-s_1)\rho.
\end{equation}
We notice that  $Q^{\lambda}_r(w)\subset Q_{s_2\rho}$ 
for $w\in E(s_1,\lambda)$ and $r\leq r_\lambda$. Then we prove a Vitali type covering 
of the super-level set $E(s_1,\lambda)$ satisfying a balancing condition on each set. 

\begin{lemma}
For each $1\leq s_1<s_2\leq 2$ and $\lambda\geq \lambda_1$, there exist $w_i\in E(s_1,\lambda)$ and $r_i\in(0,\frac{r_\lambda}{20})$, $i=1,2,3,\dots$, such that  $Q^\lambda_{4r_i}(w_i)$ are mutually disjoint,
\begin{equation}\label{stop0}
E(s_1,\lambda)\setminus N\ \subset\ \bigcup_{i=1}^\infty Q^\lambda_{20r_i}(w_i) 
\end{equation}
for some a Lebesgue measure zero set $N$,
\begin{equation}\label{stop1}
\fint_{Q^\lambda_{r_i}(w_i)}\phi(|Du|)\,dz=\phi(\lambda)
\quad\text{and}\quad
\fint_{Q^\lambda_{r}(w)}\phi(|Du|)\,dz\le \phi(\lambda) \ \ \ \text{for all }r\in(r_i,r_\lambda).
\end{equation}
\end{lemma}

\begin{proof}
For $w\in E(s_1,\lambda)$ and $r\in[\frac{r_\lambda}{20},r_\lambda)$, using \eqref{lambda0} we have
$$
\fint_{Q^\lambda_{r}(w)}\phi(|Du|)\,dz  \leq \frac{|Q_{2\rho}|}{|Q^\lambda_{r}|} \fint_{Q_{2\rho}}\phi(|Du|)\,dz \leq \frac{|Q_{2\rho}|}{|Q_{r_\lambda/20}|}\phi_2(\lambda)\mathcal{D}(\lambda_0).
$$
By \eqref{rlambda}, \eqref{Dineq2}, \eqref{lambda1} and \eqref{Dineq1}, 
%
\[\begin{split}
\frac{|Q_{2\rho}|}{|Q_{r_\lambda/20}|}\phi_2(\lambda)\mathcal{D}(\lambda_0) 
& \leq 
\bigg(\frac{40}{s_2-s_1}\bigg)^{n+2} 
\max\{1,\phi_2(\lambda)^{-\frac{n+2}{2}}\} \phi_2(\lambda)\mathcal{D}(\lambda_0)\\
&\leq \mathcal{D}(\lambda)  \max\{1,\phi_2(\lambda)^{-\frac{n+2}{2}}\} \phi_2(\lambda) \\
&= \min\{1,\phi_2(\lambda)^{\frac{n+2}{2}}\}\lambda^2\max\{1,\phi_2(\lambda)^{-\frac{n+2}{2}}\} \phi_2(\lambda)
= \phi(\lambda).
\end{split}\]
Therefore we obtain that
\[
\fint_{Q^\lambda_{r}(w)}\phi(|Du|)\,dz\leq \phi(\lambda)\quad \text{for all }
\ r\in[\tfrac{r_\lambda}{20},r_\lambda). 
\]
On the other hand, by Lebesgue's differentiation theorem we see that for almost every $w\in E(s_,\lambda)$ 
\[
\lim_{r\to0^+}\fint_{Q^\lambda_{r}(w)}\phi(|Du|)\,dz = \phi(|Du(w)|)> \phi(\lambda).
\]
Therefore, since the map $r\mapsto \fint_{Q^\lambda_{r}(w)}\phi(|Du|)\,dz$ is continuous, one can find $r_w\in (0,\frac{r_\lambda}{20})$ such that
\[
\fint_{Q^\lambda_{r_w}(w)}\phi(|Du|)\,dz=\lambda,
\quad\text{and}\quad
\fint_{Q^\lambda_{r}(w)}\phi(|Du|)\,dz\le \lambda\quad \text{for all }\ r\in(r_w,r_\lambda].
\]
Consequently, applying Vitali's covering lemma for $\{Q^\lambda_{r_w}(w)\}$, we have the conclusion.
\end{proof}

\textit{Step 3.}
By \eqref{stop1}, we can apply Lemma~\ref{lem:reverse}, so that we have that for sufficiently small $\delta\in(0,1)$,
\[\begin{split}
\phi(\lambda)
&= \fint_{Q^\lambda_{r_i}(w_i)}\phi(|Du|)\,dz \leq   c\bigg(\fint_{Q^\lambda_{4r_i}(w_i)}\phi(|Du|)^{\theta}\,dz\bigg)^{\frac{1}{\theta}}\\
& \leq c \phi(\delta\lambda)  +c\bigg(\frac{1}{|Q^\lambda_{4r_i}|}\int_{Q^\lambda_{4r_i}(w_i)\cap E(s_2,\delta\lambda)}\phi(|Du|)^{\theta}\,dz\bigg)^{\frac{1}{\theta}}\\
& \leq c \delta^{p} \phi(\lambda) + \frac{c}{|Q^\lambda_{4r_i}|}\int_{Q^\lambda_{4r_i}(w_i)\cap E(s_2,\delta\lambda)}\phi(|Du|)^{\theta}\,dz \,\bigg(\fint_{Q^\lambda_{4r_i}}\phi(|Du|)\,dz\bigg)^{1-\theta}\\
& \leq \frac12 \phi(\lambda)+c\frac{\phi(\lambda)^{1-\theta}}{|Q^\lambda_{4r_i}|}\int_{Q^\lambda_{4r_i}(w_i)\cap E(s_2,\delta\lambda)}\phi(|Du|)^{\theta}\,dz .
\end{split}\]
Then we absorb $\phi(\lambda)$ into the left-hand side. 
Using \eqref{stop1} again we have   
$$
\fint_{Q^\lambda_{20r_i}(w_i)}\phi(|Du|)\,dz 
\le \phi(\lambda) 
\leq c\frac{\phi(\lambda)^{1-\theta}}{|Q^\lambda_{4r_i}|}\int_{Q^\lambda_{4r_i}(w_i)\cap E(s_2,\delta\lambda)}\phi(|Du|)^{\theta}\,dz 
$$
and so
$$
\int_{Q^\lambda_{20r_i}(w_i)}\phi(|Du|)\,dz \leq c\phi(\lambda)^{1-\theta}\int_{Q^\lambda_{4r_i}(w_i)\cap E(s_2,\delta\lambda)}\phi(|Du|)^{\theta}\,dz .
$$
Therefore, since $\{Q^\lambda_{20r_i}(w_i)\}$ is a covering of  $E(s_1,\lambda)$ according to 
\eqref{stop0} and $Q^\lambda_{4r_i}(w_i)$ are mutually disjoint,
\[
\int_{E(s_1,\lambda)}\phi(|Du|)\,dz   \leq \sum_{i=1}^\infty \int_{Q^\lambda_{20r_i}(w_i)}\phi(|Du|)\,dz\leq c\phi(\lambda)^{1-\theta}\int_{E(s_2,\delta\lambda)}\phi(|Du|)^{\theta}\,dz.
\]
In addition, 
\[\begin{split}
\int_{E(s_1,\delta\lambda)\setminus E(s_1,\lambda)}\phi(|Du|)\,dz &\leq  \int_{E(s_1,\delta\lambda)\setminus E(s_1,\lambda)}\phi(\lambda)^{1-\theta} \phi(|Du|)^{\theta}\,dz\\
&\leq  \int_{E(s_2,\delta\lambda)}\phi(\lambda)^{1-\theta} \phi(|Du|)^{\theta}\,dz. 
\end{split}\]
Combining these and replacing $\delta\lambda$ by $\lambda$, we have 
\begin{equation}\label{15}
\int_{E(s_1,\lambda)}\phi(|Du|)\,dz\leq  c\int_{E(s_2,\lambda)}\phi(\lambda)^{1-\theta} \phi(|Du|)^{\theta}\,dz\ \ \ \text{for all }\lambda>\delta\lambda_1.
\end{equation}

\textit{Step 4.}
Let us set 
$$
|Du|_k:=\min\{|Du|,k\} \ \ \ \text{for }k\geq0,
$$
$$
E_k(s,\lambda):=\{z\in Q_{s\rho}:|Du|_k(z)>\lambda\}.
$$
 From now on, we assume that $k>\lambda_1$. Then we have from \eqref{15} that for  $\epsilon>0$, which will be determined later,
\[\begin{split}
I&:=\int_{\lambda_1}^\infty \phi(\lambda)^{\epsilon-1}\phi'(\lambda)\int_{E_k(s_1,\lambda)}\phi(|Du|_k)^{1-\theta} \phi(|Du|)^{\theta}\,dz\,d\lambda \\
&\leq \int_{\lambda_1}^k \phi(\lambda)^{\epsilon-1}\phi'(\lambda)\int_{E(s_1,\lambda)} \phi(|Du|)\,dz\,d\lambda \\
&\leq  c\int_{\lambda_1}^\infty \int_{E_k(s_2,\lambda)}\phi(\lambda)^{\epsilon-\theta}\phi'(\lambda) \phi(|Du|)^{\theta}\,dz\,d\lambda =:I\!I,
\end{split}\]
where we have used the facts that $E_k(s,\lambda)=\emptyset$ if $\lambda>k$ and $E(s,\lambda)=E_k(s,\lambda)$ if $
 \lambda\leq k$.  We then apply Fubini's theorem to $I$ and $I\!I$, so that
\[\begin{split}
 I &= \int_{E_k(s_1,\lambda_1)} \phi(|Du|_k)^{1-\theta} \phi(|Du|)^{\theta} \int_{\lambda_1}^{|Du|_k} \phi(\lambda)^{\epsilon-1}\phi'(\lambda)\, d\lambda\, dz\\
 &= \frac{1}{\epsilon} \int_{E_k(s_1,\lambda_1)} \left[\phi(|Du|_k)^{1-\theta+\epsilon} \phi(|Du|)^{\theta}-\phi(\lambda_1)^{\epsilon}\phi(|Du|_k)^{1-\theta}\phi(|Du|)^{\theta}\right]\,dz
\end{split}\]
and
\[\begin{split}
I\!I&= c\int_{E_k(s_2,\lambda_1)}\phi(|Du|)^{\theta}  \int_{\lambda_1}^{|Du|_k} \phi(\lambda)^{\epsilon-\theta}\phi'(\lambda) \,d\lambda \,dz\\
&= \frac{c}{1-\theta+\epsilon}\int_{E_k(s_2,\lambda_1)}  \left(\phi(|Du|_k)^{1-\theta+\epsilon} -\phi(\lambda_1)^{1-\theta+\epsilon}\right)\phi(|Du|)^{\theta} \, dz\\
&\leq \frac{c}{1-\theta}\int_{E_k(s_2,\lambda_1)} \phi(|Du|_k)^{1-\theta+\epsilon}   \phi(|Du|)^{\theta} \, dz.
\end{split}\]  
Therefore we have
  \[\begin{split}
\int_{E_k(s_1,\lambda_1)} \phi(|Du|_k)^{1-\theta+\epsilon} \phi(|Du|)^{\theta}\,dz &\leq   \phi(\lambda_1)^{\epsilon} \int_{E_k(s_1,\lambda_1)} \phi(|Du|_k)^{1-\theta} \phi(|Du|)^{\theta}\,dz \\
&\qquad +  c\epsilon \int_{Q_{s_2\rho}}\phi(|Du|_k)^{1-\theta+\epsilon} \phi(|Du|)^{\theta}   \, dz.
\end{split}\]  
At this stage, we choose $\epsilon=\epsilon(n,N,p,q,L,\nu,\Lambda)>0$ so small that 
$c\epsilon\leq \frac12$. On the other hand,   
\[
\int_{Q_{s_1\rho}\setminus E_k(s_1,\lambda_1)} \phi(|Du|_k)^{1-\theta+\epsilon} \phi(|Du|)^{\theta}\,dz\leq \phi(\lambda_1)^{\epsilon}\int_{Q_{s_1\rho}} \phi(|Du|_k)^{1-\theta} \phi(|Du|)^{\theta}\,dz.
\]
Combining the last two estimates, we have 
  \[\begin{split}
\int_{Q_{s_1\rho}} \phi(|Du|_k)^{1-\theta+\epsilon} \phi(|Du|)^{\theta}\,dz &\leq \frac12 \int_{Q_{s_2\rho}}\phi(|Du|_k)^{1-\theta+\epsilon} \phi(|Du|)^{\theta}   \, dz \\
& +  \frac{c\phi(\lambda_0)^{\epsilon} }{(s_2-s_1)^{\alpha_0}}\int_{Q_{2\rho}} \phi(|Du|_k)^{1-\theta} \phi(|Du|)^{\theta}\,dz,
\end{split}\]  
where $\alpha_0:=\epsilon q(n+2)\max\{\frac{1}{2},\frac{2}{p(n+2)-2n}\}$; here we used \eqref{lambda1} which yields $\phi(\lambda_1)\leq \frac{c}{(s_2-s_1)^{\alpha_0}}\phi(\lambda_0)$.
Applying the standard iteration lemma~\ref{lem:iteration} to this inequality, we find that 
\[\begin{split}
\int_{Q_{\rho}} \phi(|Du|_k)^{1-\theta+\epsilon} \phi(|Du|)^{\theta}\,dz 
&\leq c\phi(\lambda_0)^{\epsilon} \int_{Q_{2\rho}} \phi(|Du|_k)^{1-\theta} \phi(|Du|)^{\theta}\,dz\\
&\leq c\phi(\lambda_0)^{\epsilon} \int_{Q_{2\rho}} \phi(|Du|)\,dz,
\end{split}\]
Finally, letting $k\to \infty$ and recalling \eqref{lambda0}, we have 
\[\begin{split}
\fint_{Q_{\rho}}  \phi(|Du|)^{1+\epsilon}\,dz & \leq c\phi(\lambda_0)^{\epsilon } \fint_{Q_{2\rho}} \phi(|Du|)\,dz\\
&\leq  c \bigg[ \phi \bigg(\mathcal{D}^{-1}\bigg(\fint_{Q_{2\rho}}\phi(|Du|)\,dz\bigg)\bigg)\bigg]^\epsilon\fint_{Q_{2\rho}}\phi(|Du|)\,dz.
\end{split}\]




\bibliographystyle{amsplain}

\end{document}